\documentclass[10pt]{article}

\usepackage{latexsym}
\usepackage{amsmath,amsthm}
\usepackage{amssymb}
\usepackage{bm}
\usepackage{graphicx}
\usepackage{wrapfig}
\usepackage{fancybox}
\usepackage{hyperref}
\usepackage{color}
\usepackage{chngcntr}

\makeatletter
\def\@seccntformat#1{\@ifundefined{#1@cntformat}%
   {\csname the#1\endcsname\quad}  
   {\csname #1@cntformat\endcsname}
}
\let\oldappendix\appendix 
\renewcommand\appendix{%
    \oldappendix
    \newcommand{\section@cntformat}{\appendixname~\thesection:\quad}
}
\makeatother

\newtheorem{lemma}{Lemma}[section]
\newtheorem{theorem}{Theorem}[section]

\usepackage[margin=1in, centering]{geometry}
\usepackage{fancyhdr}
\usepackage[absolute]{textpos}
\TPMargin{5pt}
\newcommand{\copyrightstatement}{
\begin{textblock}{0.4}(0.08,0.93)    
\noindent
\footnotesize
\copyright Hisashi Kobayashi, 2016.
\end{textblock}
}

\begin{document}
\copyrightstatement
\title{Local Extrema of the $\Xi(t)$ Function and\\
The Riemann Hypothesis}
\date{2016/03/04}
\author{Hisashi Kobayashi\footnote{Professor Emeritus, Department of Electrical Engineering, School of Engineering and Applied Science, Princeton University, Princeton, NJ 08544.}}

\maketitle

\begin{abstract}
In the present paper we obtain a necessary and sufficient condition to prove the Riemann hypothesis in terms of certain properties of local extrema of the function $\Xi(t)=\xi(\tfrac{1}{2}+it)$.  First, we prove that positivity of all local maxima and negativity of all local minima of $\Xi(t)$ form a necessary condition for the Riemann hypothesis to be true.  After showing that any extremum point of $\Xi(t)$ is a saddle point of the function $\Re\{\xi(s)\}$, we prove that the above properties of local extrema of $\Xi(t)$ are also a sufficient condition for the Riemann hypothesis to hold at $t\gg 1$.

We present a numerical example to illustrate our approach towards a possible proof of the Riemann hypothesis. Thus, the task of proving the Riemann hypothesis is reduced to the one of showing the above properties of local extrema of $\Xi(t)$.\\

\end{abstract}

\section{Properties of local extrema of $\mathbf{\Xi(t)}$ as a necessary condition for the Riemann hypothesis}

We begin with recapitulating some definitions and relations presented in our earlier paper \cite{report-no-5}.  Riemann's $\xi(s)$ function\footnote{The $\xi$ function should be attributed to Riemann's original paper \cite{riemann}, but his function $\xi(t)$ is what Titchmarsh \cite{titchmarsh} and others denote as $\Xi(t)$, the notation we adopt here.} is defined (see e.g., Edwards \cite{edwards} p. 16) by
\begin{align}\label{xi-def}
\xi(s)=\frac{s(s-1)}{2}\pi^{-s/2}\Gamma(s/2)\zeta(s),
\end{align}
where $\zeta(s)$ is Riemann's zeta function defined by
\begin{align}
\zeta(s)=\sum_{n=1}^\infty n^{-s}, ~~\mbox{for}~~\Re(s)>1,
\end{align}
which is then extended to the entire domain of the complex variable $s=\sigma+it$ by analytic continuation (See Riemann \cite{riemann} and Edwards \cite{edwards}).

The function $\xi(s)$ has the product-form representation (see \cite{edwards}, p. 20. also Eqn. (24) of \cite{report-no-5}):
\begin{align}\label{product-form}
\xi(s)=\tfrac{1}{2}\prod_n\left(1-\frac{s}{\rho_n}\right),
\end{align}
The value of $\xi(s)$ on the critical line $s=\tfrac{1}{2}+it$ is a real function of $t$, denoted here as $\Xi(t)$ (see e.g., Titchmarsh\cite{titchmarsh}).
\begin{align}\label{def-Xi}
\Xi(t)=\xi\left(\tfrac{1}{2}+it\right)=\Re\left\{\xi\left(\tfrac{1}{2}+it\right)\right\}.
\end{align}

In the present paper we make use of the real functions $a(t)$ and $b(t)$ that we introduced in (32) and (33) of \cite{report-no-5}.
\begin{align}\label{def-a-t}
a(t)=\tfrac{1}{2}\left.\frac{\partial^2\Re\{\xi(\sigma+it)\}}{\partial\sigma^2}\right|_{\sigma=\tfrac{1}{2}}
=-\tfrac{1}{2}\Xi''(t),
\end{align}
and
\begin{align}\label{def-b-t}
b(t)&=\left.\frac{\partial\Im\{\xi(s)\}}{\partial\sigma}\right|_{\sigma=\tfrac{1}{2}}= -\Xi'(t).
\end{align}
In the subsection below, we will derive essential properties of local extrema of $\Xi(t)$ under the assumption that the Riemann hypothesis is true.

\subsection{A necessary condition for the Riemann hypothesis to be true}\label{subsec-RH-true}

Let us suppose that the Riemann hypothesis is true, i.e., all zeros of $\xi(s)$ take the form $\rho_n=\tfrac{1}{2}+it_n$, with $t_n$ being real.  Then, their complex conjugates $\rho_n^*=\tfrac{1}{2}-it_n$ must be also zeros.  We label these countably infinite zeros so that $\rho_{-n}=\rho_n^*$, with $n>0$. Then, the product form (\ref{product-form}) can be rewritten as
\begin{align}\label{product-RH-true}
\xi(s)&=\tfrac{1}{2}\prod_{n>0}\left(1-\frac{s}{\rho_n}\right)\left(1-\frac{s}{\rho_n^*}\right)=\tfrac{1}{2}\prod_{n>0}g_n(s),
\end{align}
where we define $g_n(s)$ by
\begin{align}
g_n(s)=\left(1-\frac{s}{\rho_n}\right)\left(1-\frac{s}{\rho_n^*}\right)=\frac{(s-\rho_n)(s-\rho_n^*)}{|\rho_n|^2}=\frac{\lambda^2+t_n^2-t^2+2it\lambda}{t_n^2+\tfrac{1}{4}},
\end{align}
where $\lambda=\sigma-\tfrac{1}{2}$, as defined in \cite{report-no-5}, i.e., $s=\tfrac{1}{2}+\lambda+it$. We term $g_n(s)$ the ``$n$th factor'' of the product-form expression for $\xi(s)$.
Let us bring the $j$-th factor out of the product-form:
\begin{align}\label{jth-term}
\xi(s)=g_j(s) \xi_{(j)}(s),
\end{align}
where
\begin{align}\label{xi-minus-j}
\xi_{(j)}(s)=\tfrac{1}{2}\prod_{n\neq j, n>0}g_n(s).
\end{align}
It is shown in Appendix A that $\xi_{(j)}(s)$ can be approximated by the following expression for $t\approx t_j$:
\begin{align}\label{xi_j^*-formula}
\xi_{(j)}(s)=\alpha_j(\lambda^2-t^2+\gamma_j +2it\lambda),
\end{align}
where $\alpha_j$ and $\gamma_j$ are constants.
Then, the real part of $\xi(s)$ is found to be
\begin{align}\label{Re-xi}
\Re\{\xi(s)\}&=\frac{\lambda^2+t_j^2-t^2}{t_j^2+\tfrac{1}{4}}\cdot\Re\{\xi_{(j)}(s)\}
-\frac{2\lambda t}{t_j^2+\tfrac{1}{4}}\cdot \Im\{\xi_{(j)}(s)\}
\approx -\frac{4\alpha_jt^2\lambda^2}{t_j^2},~~\mbox{for}~~ t\approx t_j\gg 1,
\end{align}
where the approximation was obtained by noting that $\lambda^2+t_j^2-t^2 \ll 2\lambda t$ for $t\approx t_j\gg 1$.  Thus, it is a parabola of $\lambda$ and touches the horizontal axis at its extremum point $\lambda=0$ (see e.g., Figure \ref{fig-cross}(a)).
Similarly, the imaginary part of $\xi(s)$ is found to be
\begin{align}\label{Im-xi}
\Im\{\xi(s)\}&=\frac{2\lambda t}{t_j^2+\tfrac{1}{4}}\cdot\Re\{\xi_{(j)}(s)\} +\frac{\lambda^2+t_j^2-t^2}{t_j^2+\tfrac{1}{4}}\cdot\Im\{\xi_{(j)}(s)\}
\approx \frac{2\alpha_jt(\gamma_j-t^2+\lambda^2)\lambda}{t_j^2},~~\mbox{for}~~ t\approx t_j\gg 1.
\end{align}
Because both real and imaginary parts of the cross section $\xi(\tfrac{1}{2}+\lambda+it_j)$ are zero at $\lambda=0$, the point $s=\tfrac{1}{2}+it_j$ is confirmed to be a zero of $\xi(s)$, as it should be.  For a further discussion on the cross-sections, see Appendix A and the numerical example in the next subsection, .

By setting $\lambda=0$ in (\ref{product-RH-true}), we find
\begin{align}
\Xi(t)=\tfrac{1}{2}\prod_{n>0}\frac{t_n^2-t^2}{t_n^2+\tfrac{1}{4}}.
\end{align}
Taking the logarithm and differentiating both sides, we have
\begin{align}
\frac{\Xi'(t)}{\Xi(t)}&=2\sum_{n>0}\frac{t}{t^2-t_n^2}.
\end{align}
Differentiating the above once more, we obtain
\begin{align}\label{diffs-of-txi}
\frac{\Xi''(t)}{\Xi(t)}- \left(\frac{\Xi'(t)}{\Xi(t)}\right)^2
&= -2\sum_{n>0}\frac{t^2+t_n^2}{(t^2-t_n^2)^2}<0 ,~~\mbox{for all}~~t,
\end{align}
which can be rearranged, using $a(t)$ and $b(t)$ defined in (\ref{def-a-t}) and (\ref{def-b-t}), as
\begin{align}
\frac{a(t)}{\Xi(t)}=-\frac{1}{2}\left(\frac{b(t)}{\Xi(t)}\right)^2+\sum_{n>0}\frac{t^2+t_n^2}{(t^2-t_n^2)^2},
\end{align}
which leads us to the following theorem:
\begin{theorem}(Local extrema of the function $\Xi(t)$ and the Riemann hypothesis)\label{theorem-necessity-RH}

If the Riemann hypothesis is true, local maxima of the function $\Xi(t)$ are all positive, and local minima are all negative.
\end{theorem}
\begin{proof}
A local extremum of $\Xi(t)$ occurs at $t=t^*$ such that $\Xi'(t^*)=0$.  Then, from (\ref{diffs-of-txi}) we find that the following relation holds at all extremum points $t^*$:
\begin{align}\label{2nd-dif-txi-to-txi}
\frac{\Xi''(t^*)}{\Xi(t^*)}&=-2\sum_{n>0}\frac{{t^*}^2+t_n^2}{({t^*}^2-t_n^2)^2}<0,
\end{align}
which implies that if $\Xi(t^*)$ is a local maximum (i.e., if $\Xi''(t^*)<0$), then $\Xi(t^*)>0$.  Likewise, if $\Xi(t^*)$ is a local minimum (i.e., if $\Xi''(t^*)>0$), then $\Xi(t^*)<0$.  Thus, we have proved the theorem.
\end{proof}

It is important to note that the inequality in (\ref{2nd-dif-txi-to-txi}) does not depend on the location of the extremum points $t^*$ vis-\`{a}-vis the zeros $t_n$'s. The above lemma asserts that positiveness of all local maxima and negativeness of all local minima of $\Xi(t)$ form a necessary condition for the Riemann hypothesis to be true. \\

\subsection{An illustrative example}
\begin{figure}
\centering
\begin{minipage}{.5\textwidth}
  \centering
  \includegraphics[scale=0.7]{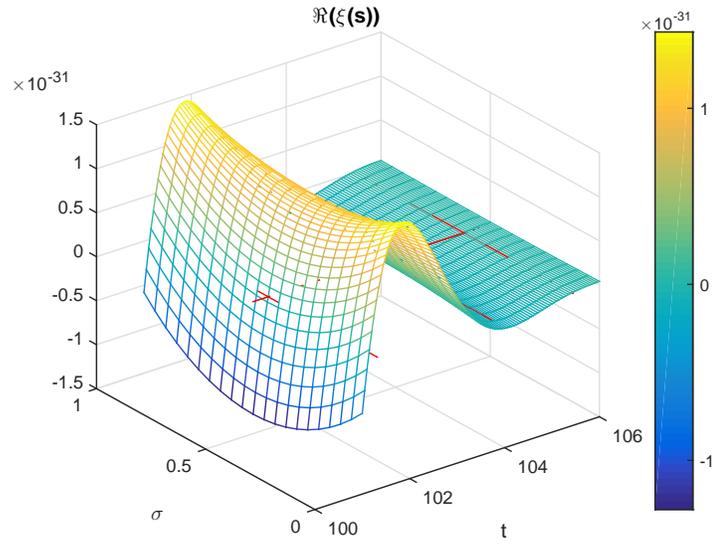}\\
   {\sf(a)}
 \end{minipage}
 \begin{minipage}{.5\textwidth}
  \centering
  \includegraphics[scale=0.7]{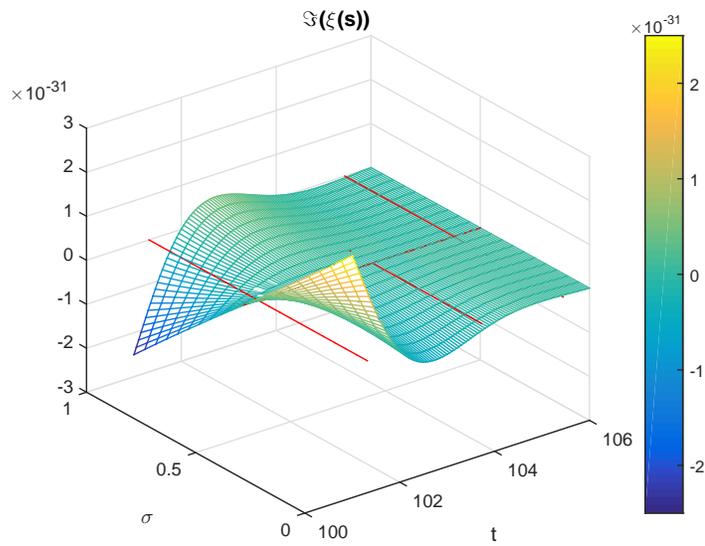}\\
   {\sf(b)}
  \end{minipage}
\caption{\sf (a) $\Re(\xi(\sigma+it))$ and (b) $\Im(\xi(\sigma+it))$,  for $0<\sigma< 1.0$ and $101\leq t\leq 106$.}
\label{fig-surf}
\end{figure}
In this subsection, we present some numerical example. Let us consider the range $101\leq t\leq 106$, in which there are three zeros: $t_{30}=101.3178\ldots, t_{31}=103.7255\ldots$ and $t_{32}=105.4466\ldots$.  Figure \ref{fig-surf} (a) and (b) are surface plots of the real and imaginary parts of the function $\xi(s)$ in the critical zone $0<\sigma<1$ (or equivalently $-\tfrac{1}{2}<\lambda<\tfrac{1}{2}$).  Here, the critical line $\sigma=\tfrac{1}{2}$ (i.e., $\lambda=0$) and the three lines $t=t_{30}$, $t=t_{31}$ and $t=t_{32}$ that pass through these zeros are shown in magenta.  The magnitude of the function is in the order $O(10^{-31})$ near $t=t_{31}$, but it rapidly decreases down to $O(10^{-32})$ and $O(10^{-33})$ at $t=t_{31}$ and $t=t_{32}$, respectively.

In Figure \ref{fig-Xi} (a), we show a plot of the function $\Xi(t)$, which clearly crosses zero at $t_{30}$, $t_{31}$ and $t_{32}$, and local maxima are positive and local minima are negative.  In Figure \ref{fig-Xi} (b), we plot $a(t)=-\tfrac{1}{2}\Xi''(t)$ (in green) and $b(t)=-\Xi'(t)$ (in red) together with $\Xi(t)$ (in blue).
\begin{figure}
\centering
\begin{minipage}{.5\textwidth}
  \centering
  \includegraphics[scale=0.7]{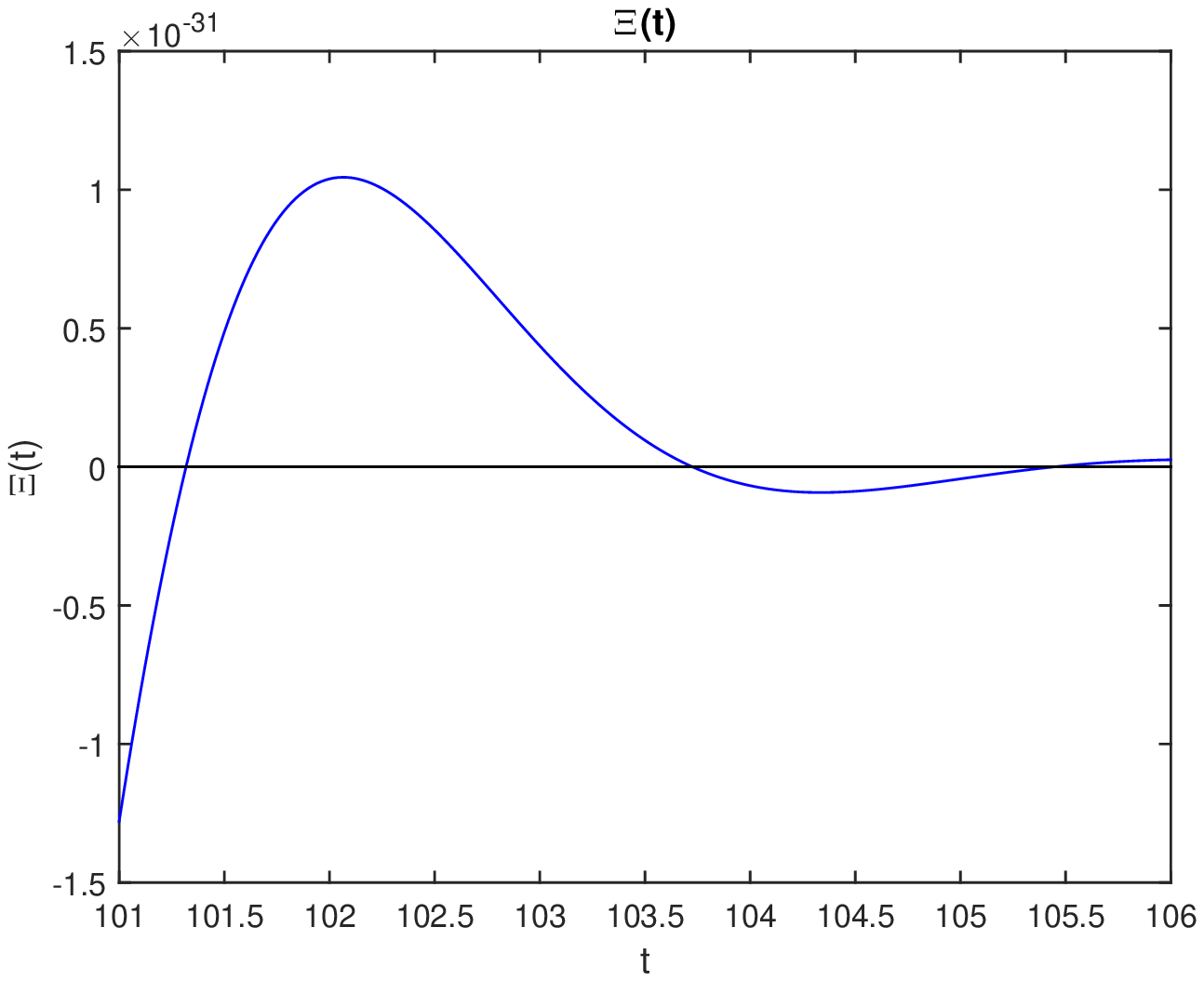}\\
   {\sf(a)}
 \end{minipage}
 \begin{minipage}{.5\textwidth}
  \centering
  \includegraphics[scale=0.7]{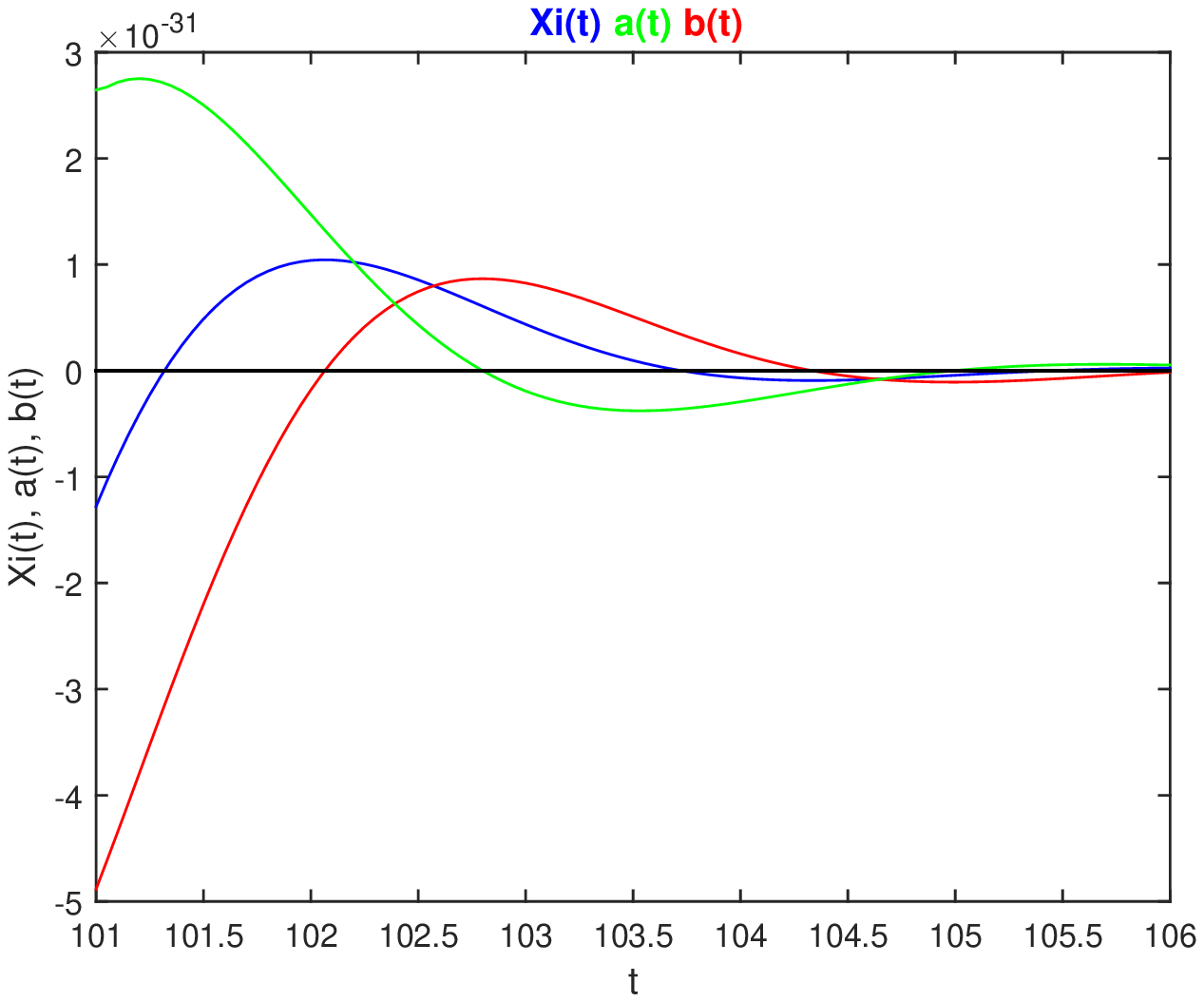}\\
   {\sf(b)}
  \end{minipage}
\caption{\sf (a) $\Re(\xi(\sigma+it))$ and (b) $\Im(\xi(\sigma+it))$,  for $0<\sigma< 1.0$ and $101\leq t\leq 106$.}
\label{fig-Xi}
\end{figure}

Figure \ref{fig-cross} (a) shows the cross-section of $\xi(\sigma+it_{30})$ on the line $t=t_{30}=101.3178\ldots$ across the critical strip $0<\sigma<1$ (or $-\tfrac{1}{2}<\lambda<\tfrac{1}{2}$).  The parabola (in blue) is the real part, and the straight line (in red), the imaginary part. In the figure (b) we show the real part (in blue) and the imaginary part (in red) of $\xi_{30}(\sigma+it_{30})$, which was calculated by dividing $\xi(s)$ by the 30th factor $f_{30}(s)$ in the product-form representation (\ref{jth-term}).
The figure (c) is the real part of the cross-section of the 30th factor, and the figure (d) is the imaginary part of this factor. In (d) we plot the real part again in blue, but because it is so small (i.e., on the order of $O(10^{-5})$ compared with the imaginary part, which is $(O^{-2})$, the blue line is indistinguishable from the horizontal axis.

\begin{figure}
\centering
\begin{minipage}{.5\textwidth}
  \centering
  \includegraphics[scale=0.6]{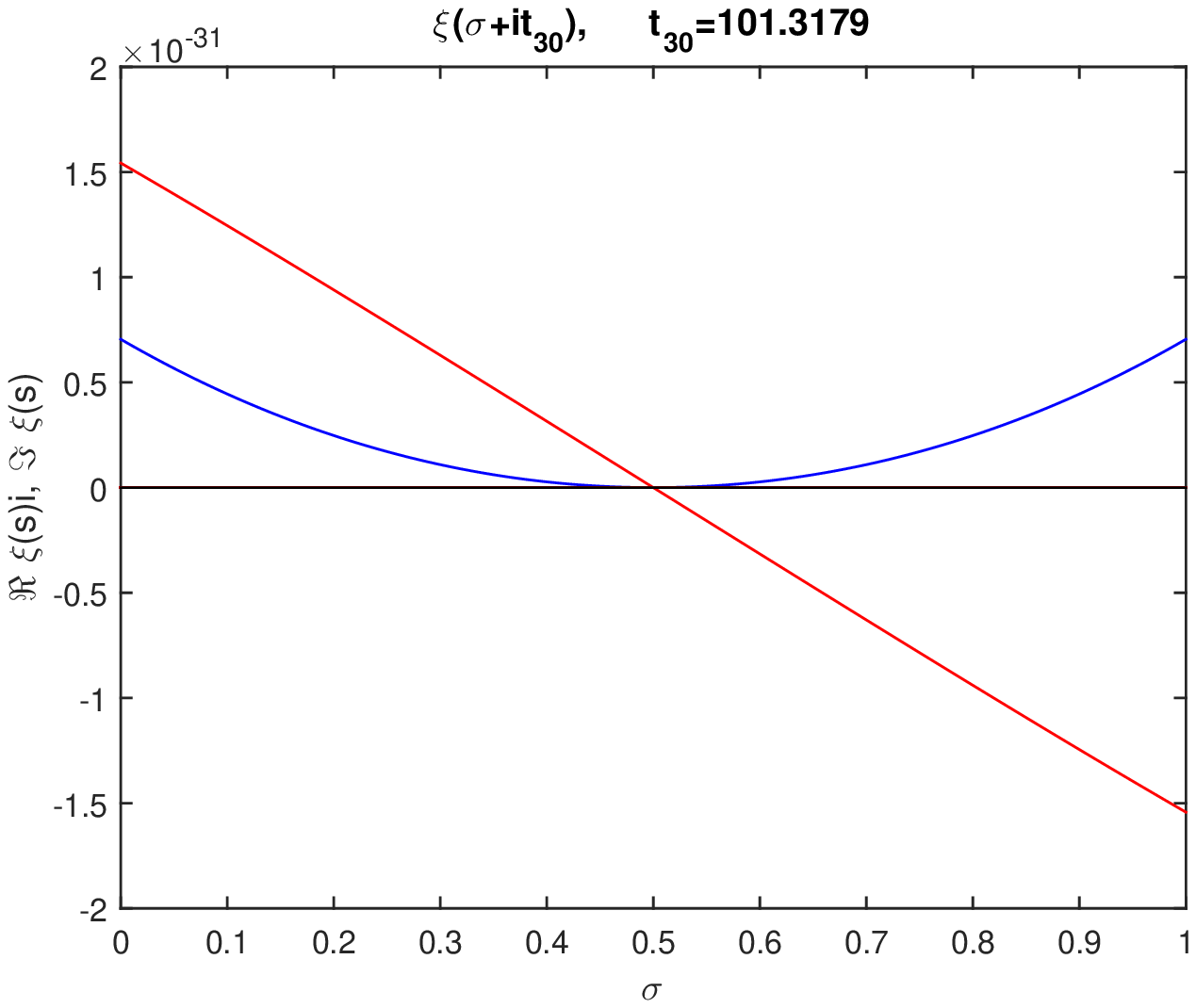}\\
   {\sf(a)}
 \end{minipage}%
 \begin{minipage}{.5\textwidth}
  \centering
  \includegraphics[scale=0.6]{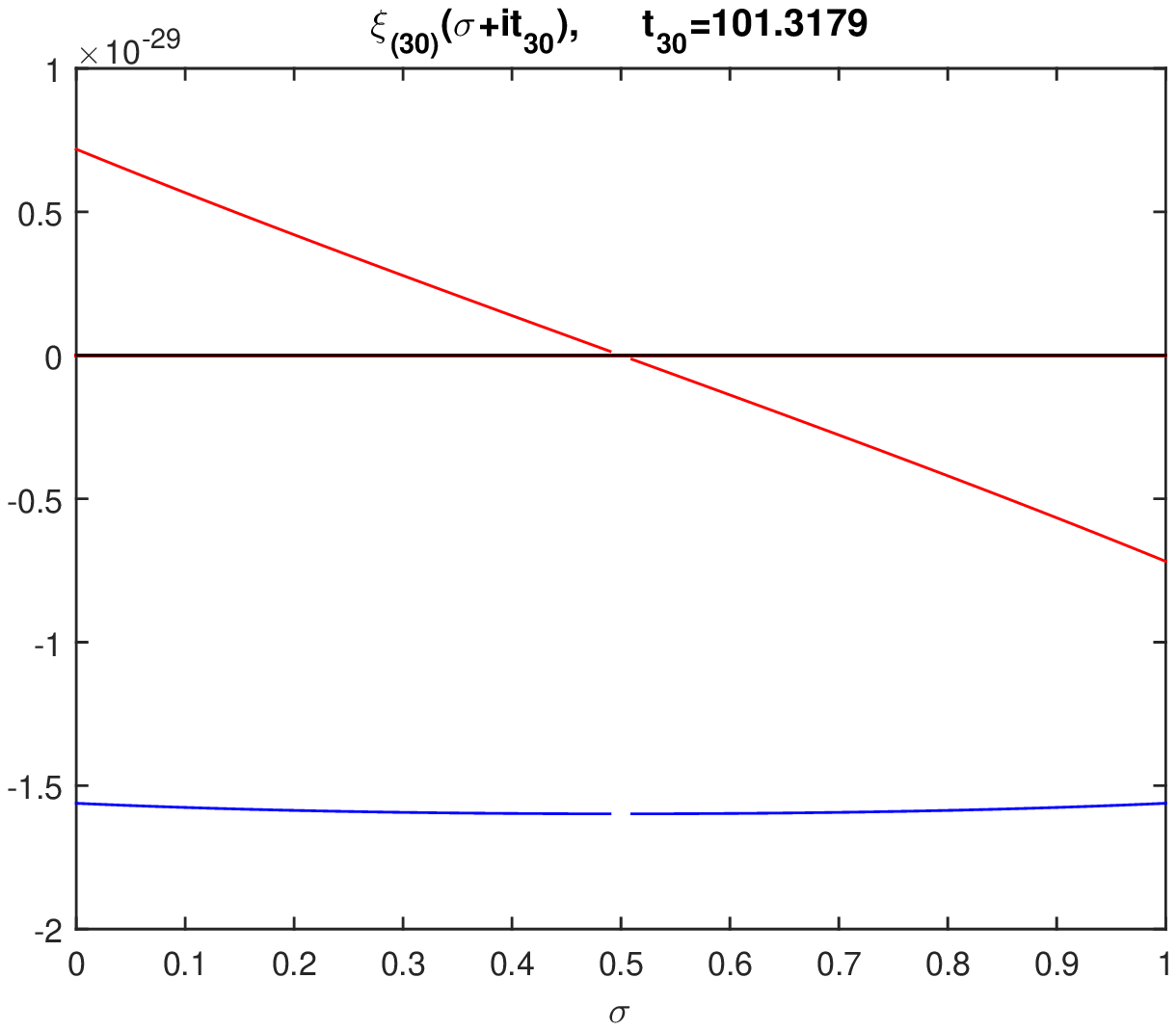}\\
   {\sf(b)}
  \end{minipage}
\begin{minipage}{.5\textwidth}
  \centering
  \includegraphics[scale=0.6]{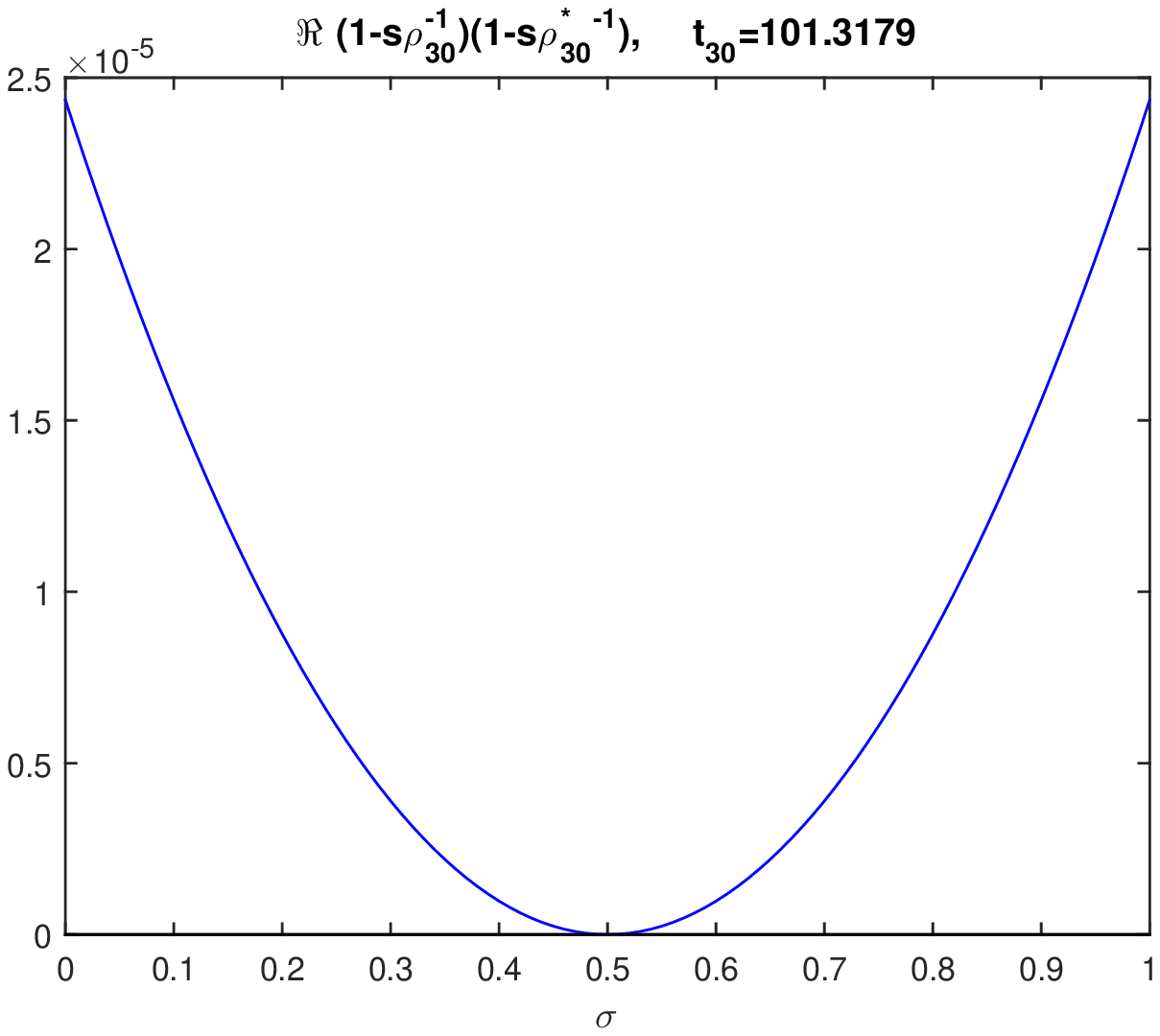}\\
   {\sf(c)}
 \end{minipage}%
 \begin{minipage}{.5\textwidth}
  \centering
  \includegraphics[scale=0.6]{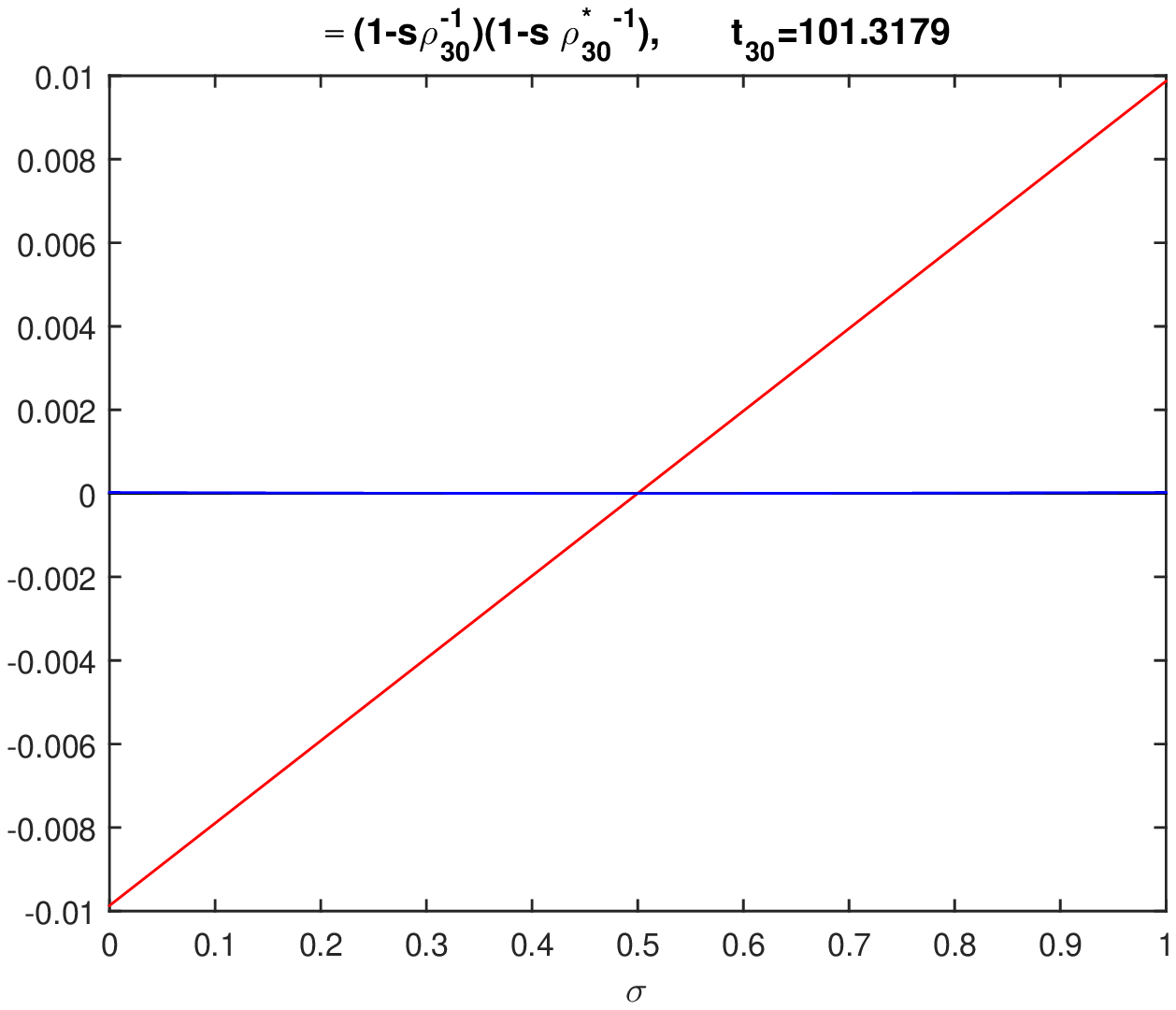}\\
   {\sf(d)}
  \end{minipage}
\caption{\sf (a) The cross section of $\xi(\sigma+it_{30})$ for $0<\sigma< 1.0$, where $t_{30}=101.3178\ldots$; (b) The cross section of $\xi_{(30)}(\sigma+it_{30})$; (c) The real part of the 30th factor $f_{30}(s)$ for $0<\sigma< 1.0$, where $s=\sigma+it_{30}$, $\rho_{30}=\tfrac{1}{2}+it_{30}$, and $\rho_{30}^*=\tfrac{1}{2}-it_{30}$; (d) The imaginary part (in red) and the imaginary part (in blue) of the 30th factor.}
\label{fig-cross}
\end{figure}
\section{Properties of local extrema of $\Xi(t)$ as a sufficient condition for the Riemann hypothesis}
In this section we shall investigate the inverse of Theorem \ref{theorem-necessity-RH}.
Let us start with the following lemma:
\begin{lemma}\label{lemma-saddle-point}(Local extrema of $\Xi(t)$ and saddle points)

Any extremum point of $\Xi(t)$ is a saddle point of the function $\Re\{\xi(s)\}$.
\end{lemma}

\begin{proof}
Since $\xi(s)$ is a holomorphic function, $\Re\{\xi(s)\}$ is a harmonic function. It is known that every critical point of a harmonic function is a saddle point. To verify this, we apply the so-called discriminant function test (or the second partial derivative test), which is well known in multivariate calculus, that is, to examine whether the determinant of the Hessian matrix of $\Re\{\xi(s)\}$ evaluated at $s=(\sigma, t)=(\tfrac{1}{2}, t^*)$ is negative or not\footnote{In this section, we often write $s=(\sigma, t)$ instead of $s=\sigma+it$ for notational convenience.}:
\begin{align}
D(\tfrac{1}{2},t^*)&=\Xi''(t^*)\cdot \left.\frac{\partial^2\Re\{\xi(\sigma,t)\}}{\partial\sigma^2}\right|_{s=(\frac{1}{2},t^*)}
 -\left(\left.\frac{\partial^2\Re\{\xi(\sigma,t)\}}{\partial\sigma\partial t}\right|_{s=(\frac{1}{2}, t^*)}\right)^2 \nonumber\\
&=-\Xi''(t^*)^2
 -\left(-\left.\frac{\partial^2\Im\{\xi(\sigma,t)\}}{\partial\sigma^2}\right|_{s=(\frac{1}{2},t^*)}\right)^2.
\end{align}
The first term of the last expression was obtained by applying Laplace's equation to the function $\Re\{\xi(\sigma, t)\}$, and the second term was obtained by using the Cauchy-Riemann equation, and this last term is zero because of (8) in \cite{report-no-5}.
Hence, we find
\begin{align}
D(\tfrac{1}{2},t^*)=-\Xi''(t^*)^2\leq 0,
\end{align}

If $\Xi'(t)$ crosses zero upward at $t=t^*$, then $\Xi''(t^*)>0$, and $\Re\{\xi(s)\}$ is a convex function of $t$ and a concave function of $\sigma$ at the point $(\tfrac{1}{2}, t^*)$. Thus, this point is a saddle point of the function $\xi(s)$.

If $\Xi'(t)$ crosses zero downward at $t=t^*$, then $\Xi''(t^*)<0$, and $\Re\{\xi(s)\}$ is a concave function of $t$ and a convex function of $\sigma$ at $(\tfrac{1}{2}, t^*)$, which is again a saddle point of $\xi(s)$.

If $\Xi''(t^*)=0$, the point $(t^*)$ is an inflection point in both
$t$ and $\sigma$ directions.  But since $\Re\{\xi(s)\}$ is symmetric around $\sigma=\tfrac{1}{2}$, we find that
$\displaystyle{\frac{\partial^3 \Re\{\xi(\sigma,t^*)\}}{\partial t^3}=0}$, which implies
$\Re\{\xi(\sigma,t^*)\}=\Re\{\xi(\tfrac{1}{2},t^*)\}=\Xi(t^*)$ for all $\sigma$.  It then follows that $\Xi''(t^*)=0$, if and only if $\Re\{\xi(\sigma,t^*)\}=\Xi(t^*)=0$ for all $\sigma$, which is impossible. Thus, we have shown that all local extremum points of $\Xi(t)$ are saddle points of $\Re\{\xi(\sigma,t)\}$.
\end{proof}

\subsection{When the Riemann hypothesis is false}

Now let us suppose that the Riemann hypothesis is false, i.e., there should exist a pair of zeros $\rho_k$ and $\rho_{k'}$ of the form
\begin{align}\label{rho_k-and rho_k'}
\rho_k=\tfrac{1}{2}+\lambda_k + it_k,~~{\rm and}~~\rho'_k=\tfrac{1}{2}-\lambda_k+ it_k,
\end{align}
and their complex conjugates
\begin{align}\label{rho_conjugates}
\rho_k^*=\rho_{-k}=\tfrac{1}{2}+\lambda_k - it_k,~~{\rm and}~~{\rho'_k}^*=\rho'_{-k}=\tfrac{1}{2}-\lambda_k- it_k,
\end{align}
must be also zeros. We assume, without loss of generality, that $t_k>0$ and $0<\lambda_k<\tfrac{1}{2}$.
Then the product-form expression for the hypothetical $\xi$-function, denoted $\xi_H(s)$, can be written as
\begin{align}\label{xi_H}
\xi_H(s)&=f_k(s)\xi_{(k)}(s),
\end{align}
where $f_k(s)$ is also a hypothetical function, representing the product of the four multiplicative factors in the product-form (\ref{product-form})\footnote{In arriving at the last line in (\ref{f_k(s)}), we used the approximation
\begin{align}
|\rho_k|^2|\rho'_k|^2&=\left[t_k^2+(\tfrac{1}{2}+\lambda_k)^2\right]\left[t_k^2+(\tfrac{1}{2}-\lambda_k)^2\right]\nonumber\\
&=t_k^4+(\tfrac{1}{2}+2\lambda_k^2)t_k^2+(\tfrac{1}{4}-\lambda_k^2)^2=t_k^4(1+O(t_k^{-2}))\approx t_k^4,
\end{align}
which should not affect the validity of the rest of this paper.}:
\begin{align}\label{f_k(s)}
f_k(s)&=\left(1-\frac{s}{\rho_k}\right)\left(1-\frac{s}{\rho'_k}\right)
\left(1-\frac{s}{\rho_k^*}\right)\left(1-\frac{s}{\rho'_{-k}}\right)\nonumber\\
&\approx t_k^{-4}[\lambda^2-\lambda_k^2-(t-t_k)^2+i2(t-t_k)\lambda]\cdot[\lambda^2-\lambda_k^2-(t+t_k)^2+i2(t+t_k)\lambda],
\end{align}
and $\xi_{(k)}(s)$, as defined by (\ref{xi-minus-j}), is obtained by removing the $k$th factor in the product-form (\ref{product-form}):\footnote{One might argue that we should remove another factor, say the $(k+1)$-st factor, since we are multiplying by $f_k(s)$, a polynomial of 4th order in $s$, in (\ref{xi-minus-j}).  But such consideration should not affect the essence of our conclusion of this section, other than $\hat{\Xi}_H(t)$ would change its polarity, and its magnitude by a factor of approximately $t_j/2$.}
\begin{align}\label{xi{(k)}}
\xi_{(k)}(s)=\tfrac{1}{2}\prod_{n\neq k,n>0}\left(1-\frac{s}{\rho_n}\right)\left(1-\frac{s}{\rho_n^*}\right).
\end{align}

As is derived in Appendix A, $\xi_{(k)}(\tfrac{1}{2}+\lambda+it_k)$ can be be given by the following estimate in the vicinity of $t\approx t_k$:
\begin{align}\label{xi-k-approx}
\hat{\xi}_{(k)}(\tfrac{1}{2}+\lambda+it_k)= \alpha_k(\lambda^2-t^2+\gamma_k+2it\lambda),~~{\rm for}~~t\approx t_k,
\end{align}
where $\alpha_k$ and $\gamma_k$ are real constants.
This approximation formula could be improved by including higher order terms $\lambda^3$ (in the imaginary part) and or $\lambda^4$ terms (in the real part).

Thus, (\ref{xi_H}) can be estimated for $t\approx t_k$ by
\begin{align}\label{approx_xi_H}
\hat{\xi}_H(s)=\alpha_k(\lambda^2-t^2+\gamma_k+2it\lambda)f_k(s),
\end{align}
from which we see that the real part is an even function of $\lambda$ and is zero exactly at $\lambda=\pm\lambda_k$, and the imaginary part is an odd function of $\lambda$ and is zero at $\lambda=\pm\lambda_k$ and $\lambda=0$ (see Appendix B).

By setting $\lambda=0$ in (\ref{xi_H}), we find
\begin{align}\label{Xi-product-form}
\Xi_H(t)=\xi_H(\tfrac{1}{2}+it)=\xi_{(k)}(\tfrac{1}{2}+it)f_k(\tfrac{1}{2}+it),
\end{align}
which can be approximated in the vicinity of $t=t_k$ by
\begin{align}\label{approx_Xi-product-form}
\hat{\Xi}_H(t)=\xi_H(\tfrac{1}{2}+it)=-\alpha_k(t^2-\gamma_k) t_k^{-4}[(t-t_k)^2+\lambda_k^2][(t+t_k)^2+\lambda_k^2],
\end{align}
and its value at $t=t_k$:
\begin{align}\label{Xi-t_k}
\hat{\Xi}_H(t_k)=-\alpha_kt_k^{-4}\lambda_k^2(4t_k^2+\lambda_k^2)(t_k^2-\gamma_k)
= -4\alpha_k\lambda_k^2t_k^{-2}(t_k^2-\gamma_k)+O(t_k^{-4}).
\end{align}

By taking the absolute value of both sides in (\ref{Xi-product-form}), taking their logarithms and differentiating them w.r.t. to $t$, we find
\begin{align}\label{log-dif-absXi}
\frac{|\hat{\Xi}_H(t)|'}{|\hat{\Xi}_H(t)|}=\frac{2(t-t_k)}{(t-t_k)^2+\lambda_k^2}+\frac{2(t+t_k)}{(t+t_k)^2+\lambda_k^2}+\frac{2t}{t^2-\gamma_k}.
\end{align}
Note that $\displaystyle{\frac{|f(t)|'}{|f(t)|}=\frac{f'(t)}{f(t)}}$ for any real function $f(t)$ at $f(t)\neq 0$. Thus, the evaluation of the above at $t=t_k$ yields
\begin{align}\label{Xi'-to-Xi}
\frac{\hat{\Xi}_H'(t_k)}{\hat{\Xi}_H(t_k)}&=\frac{4t_k}{4t_k^2+\lambda_k^2}+\frac{2t_k}{t_k^2-\gamma_k}
=t_k^{-1}\frac{(3t_k^2-\gamma_k)}{(t_k^2-\gamma_k)}+O(t_k^{-2}),
\end{align}
which, together with (\ref{Xi-t_k}), gives
\begin{align}
\hat{\Xi}_H'(t_k)=-36\alpha_k\lambda_k^2t_k^{-3}(3t_k^2-\gamma_k)+O(t_k^{-2}),
\end{align}
which should be equal to $-b(t_j)$.  Since the above is not exactly zero, the point $t=t_k$ is not a local extremum point of $\hat{\Xi}_H(t)$.

By differentiating (\ref{log-dif-absXi}) once more, we find
\begin{align}
\frac{\hat{\Xi}_H''(t)}{\hat{\Xi}_H(t)}-\left(\frac{\hat{\Xi}_H'(t)}{\hat{\Xi}_H(t)}\right)^2
&=\frac{2}{(t-t_k)^2+\lambda_k^2}-\left(\frac{2(t-t_k)}{(t-t_k)^2+\lambda_k^2}\right)^2
 +\frac{2}{(t+t_k)^2+\lambda_k^2}-\left(\frac{2(t+t_k)}{(t+t_k)^2+\lambda_k^2}\right)^2\nonumber\\
&~~ +\frac{2}{t^2-\gamma_k}-\left(\frac{2t}{t^2-\gamma_k}\right)^2+O(t_k^{-3}),
\end{align}
which leads to
\begin{align}\label{Xi''-to-Xi}
\frac{\hat{\Xi}_H''(t_k)}{\hat{\Xi}_H(t_k)}-\left(\frac{\Xi'(t_k)}{\Xi(t_k)}\right)^2+O(t_k^{-3})
 &=\frac{2}{\lambda_k^2}+\frac{2}{4t_k^2+\lambda_k^2}-\frac{16t_k^2}{(4t_k^2+\lambda_k^2)^2}
 +\frac{2}{t_k^2-\gamma_k}-\frac{4t_k^2}{(t_k^2-\gamma_k)^2}\nonumber\\
 &=\frac{2}{\lambda_k^2}+\frac{2}{4t_k^2+\lambda_k^2}+\frac{2}{t_k^2-\gamma_k}-\frac{4t_k^2}{(t_k^2-\gamma_k)^2}
+O(t_k^{-3})\nonumber\\
&=\frac{2}{\lambda_k^2}-\frac{4t_k^2}{(t_k^2-\gamma_k)^2}+O(t_k^{-3}).
\end{align}

Based on what we have found up to this point, we make the following proposition:
\begin{lemma}(A necessary condition for the Riemann hypothesis to be false)\label{lemma-conjecture}

Suppose that we conjecture that a pair of zeros possibly exist on the line $t=t_k$, together with their complex conjugates on the line $t=-t_k$.  A necessary condition for this conjecture to be true (i.e., for the Riemann hypothesis to fail) is to show that the value of $\displaystyle{\frac{\Xi''(t_k)}{\Xi(t_k)}}$ is strictly positive.
\end{lemma}
\begin{proof}
Suppose $\Xi''(t_k)>0$.  Then, the function $\Xi(t)$ is a convex function of $t$ in the vicinity of $t=t_k$. From Laplace's equation  $\displaystyle{\frac{\partial^2\Re\{\xi(\tfrac{1}{2}+\lambda+it_k)\}}{\partial\lambda^2}=2a(t_k)=-\Xi''(t_k)<0}$, thus the cross-section of $\Re\{\xi(\tfrac{1}{2}+\lambda+it_k)\}$ along $t=t_k$ is a concave function of $\lambda$, and is a parabola in the range $\displaystyle{|\lambda|<\tfrac{1}{2}}$, with its maximum occurring at $\lambda=0$.  Then, a necessary and sufficient condition for this function to become zero at $\lambda=\pm\lambda_k$, is that the cross-section at $\lambda=0$, which is $\Xi(t_k)$, is strictly positive.  If $\Xi(t_k)=0$, then it implies $\lambda_k=0$, the degenerated case where the two zeros $\pm\lambda_k$ reduce to one on the critical line.

If $\Xi''(t_k)<0$, the above argument should be modified by replacing ``convex'' to ``convex,'' ''maximum'' to ``minimum,'' and ``positive'' to ``negative.''  In either case, the ratio $\Xi''(t_k)/\Xi(t_k)$ must be strictly positive for the Riemann hypothesis to be rejected.
\end{proof}

Unfortunately, $\Xi(t_k)$ is neither a local maximum, nor a local minimum, because $\Xi'(t_k)=-b(t_k)\neq 0$.
Thus, the above lemma cannot be directly applicable to solving the Riemann hypothesis, since we have no way of knowing a possible value of $t_k$, even if such $t_k$ should exist.
As we show in Appendix B, however, there should always exist a local extremum point $t=t^*=t_k-\delta_k$, where $\delta_k$ is strictly positive, as given by (\ref{delta-in-theorem}) below.  Thus, we can make the following assertion:

\begin{theorem}(A sufficient condition for the Riemann hypothesis to hold asymptotically)\label{theorem-sufficient}

If local maxima of $\Xi(t)$ are all nonnegative and its local minima are are all non-positive, then the Riemann hypothesis is asymptotically true, i.e., at $t\gg 1$.
\end{theorem}
\begin{proof}
As is derived in Appendix B, for any $t_k$ where nontrivial zeros off the critical line might exist, there is a point $t^*=t_k-\delta_k$, which is a local extremum of the real function $\Xi(t)$, where
\begin{align}\label{delta-in-theorem}
\delta_k=\frac{3\lambda_k^2}{2t_k}+O(t_k^{-2}).
\end{align}
Thus, it is clear that $\displaystyle{\lim_{t_k\to\infty}\delta_k=0}$.
\end{proof}

\subsection{The example continued}
Let us continue the example of the preceding section.  For the sake of a hypothetical example, let us consider the case $k=30$ (i.e., $t_{30}=101.3178\ldots$), and $\lambda_k=\tfrac{1}{4}$, i.e., we have
\[ \rho_{30}=\frac{3}{4}+it_{30},~~{\rm and}~~\rho'_{30}=\frac{1}{4}+it_{30},\]
and their complex conjugates as zeros of $\xi(s)$.

Figures \ref{fig-hypothetical} (a) and (b) are the real and imaginary parts of the factor function $f_{30}(s)$ (\ref{f_k(s)}).
We find that $\alpha_{30}$ and $\gamma_{30}$ that approximate $\xi_{(30)}(\tfrac{1}{2}+\lambda+it_{30})$ plotted earlier in Figure \ref{fig-cross} (b) are estimated to be
\[ \alpha_{30}=-7.0896\times 10^{-32},~~~{\rm and}~~\gamma_{30}=1.0491\times 10^4.\]
Figure \ref{fig-hypothetical} (c) is a plot of the function (\ref{xi-k-approx}) at $t=t_k$, which is hardly distinguishable from the curves of Figure \ref{fig-cross} (b).  Figure \ref{fig-hypothetical}(d) is the product of the functions plotted in (a), (b) and (c).  Should this hypothetical cross-section in fact materialize, we clearly see that $\frac{1}{4}+it_{30}$ and $\frac{3}{4}+it_{30}$ could be zeros of $\xi_H(s)$.

\begin{figure}
\centering
\begin{minipage}{.5\textwidth}
  \centering
  \includegraphics[scale=0.6]{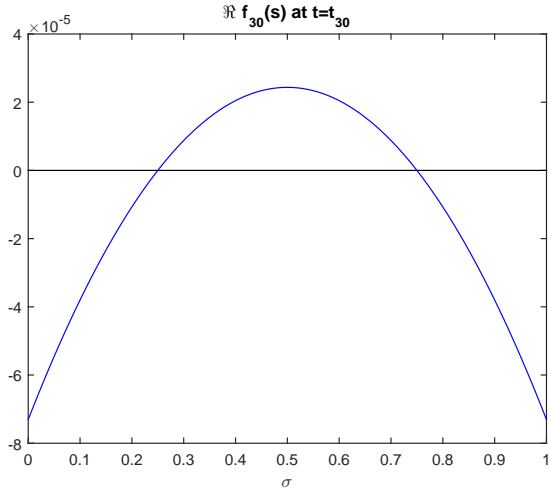}\\
   {\sf(a)}
 \end{minipage}%
 \begin{minipage}{.5\textwidth}
  \centering
  \includegraphics[scale=0.6]{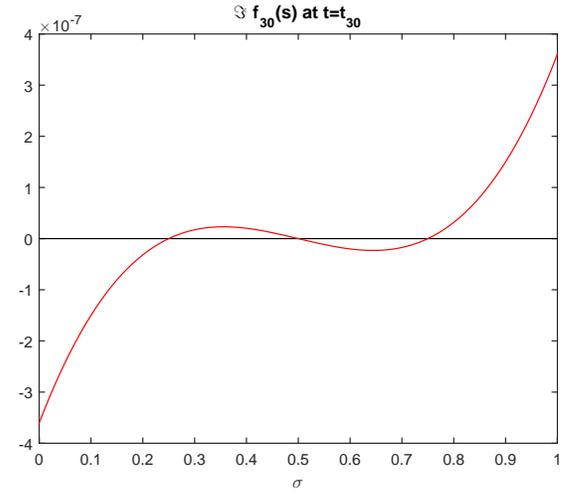}\\
   {\sf(b)}
  \end{minipage}
\begin{minipage}{.5\textwidth}
  \centering
  \includegraphics[scale=0.6]{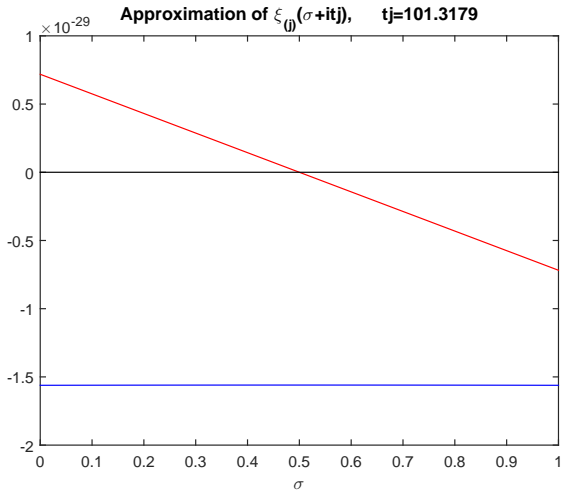}\\
   {\sf(c)}
 \end{minipage}%
 \begin{minipage}{.5\textwidth}
  \centering
  \includegraphics[scale=0.6]{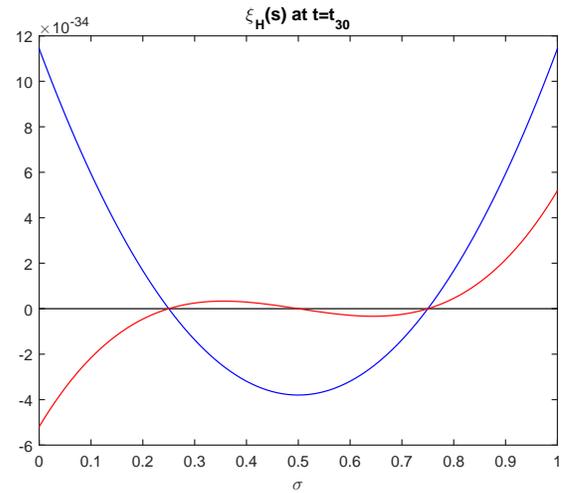}\\
   {\sf(d)}
  \end{minipage}
\caption{\sf (a) The real part of $f_k(s)$ of (\ref{f_k(s)}), $k=30$; (b) The imaginary part of $f_k(s)$ of (\ref{f_k(s)}); (c) The cross section $\alpha_k(\lambda^2-t_k^2+\gamma_k+2it_k\lambda)$ used in (\ref{xi_H}); The real part is in blue and the imaginary part is in red; (d) The cross section of $\hat{\xi}_H(s)$ of (\ref{xi_H}). The real part is in blue and the imaginary part is in red.}
\label{fig-hypothetical}
\end{figure}

\begin{figure}
\centering
\begin{minipage}{.5\textwidth}
  \centering
  \includegraphics[scale=0.6]{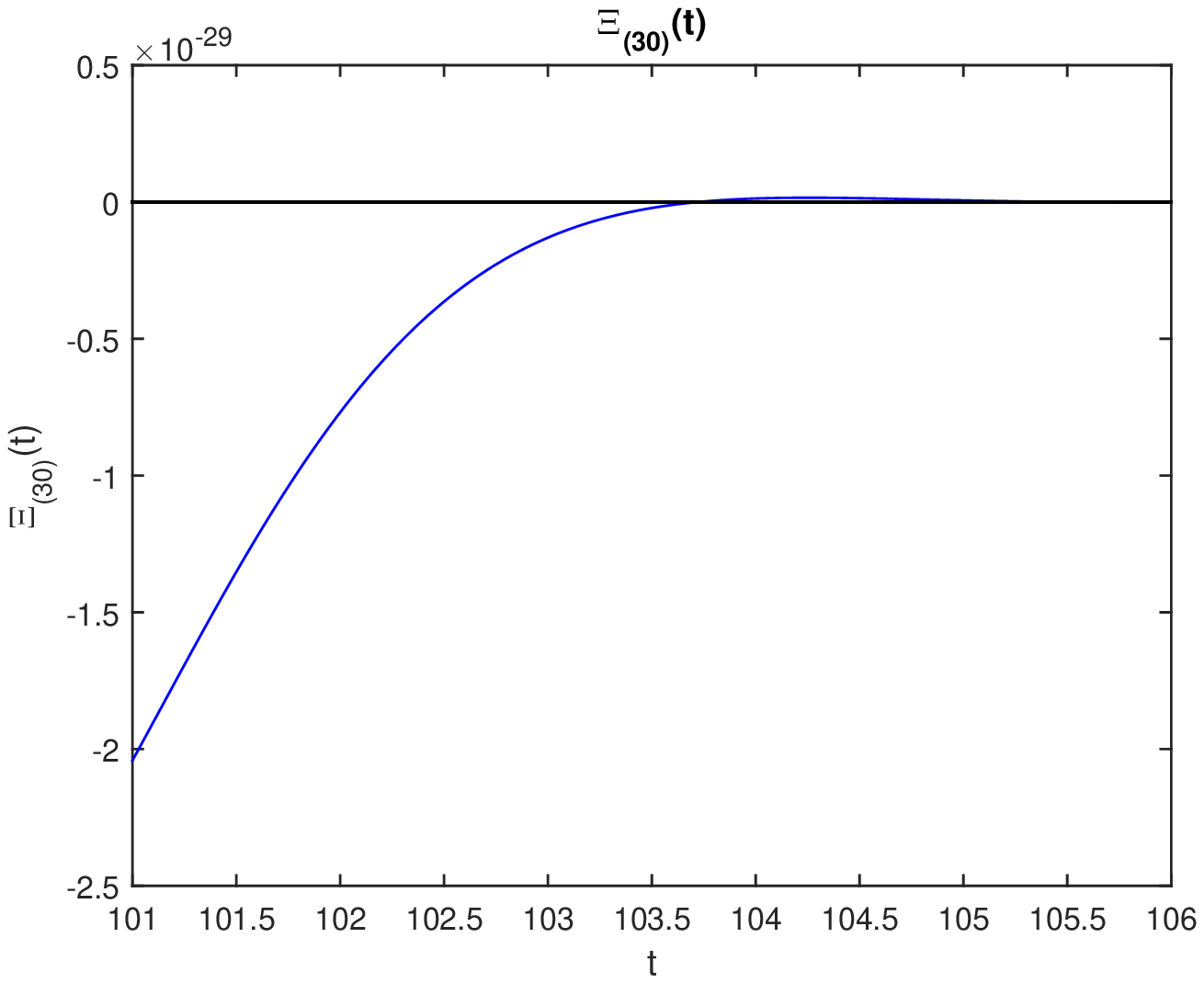}\\
   {\sf(a)}
 \end{minipage}%
 \begin{minipage}{.5\textwidth}
  \centering
  \includegraphics[scale=0.6]{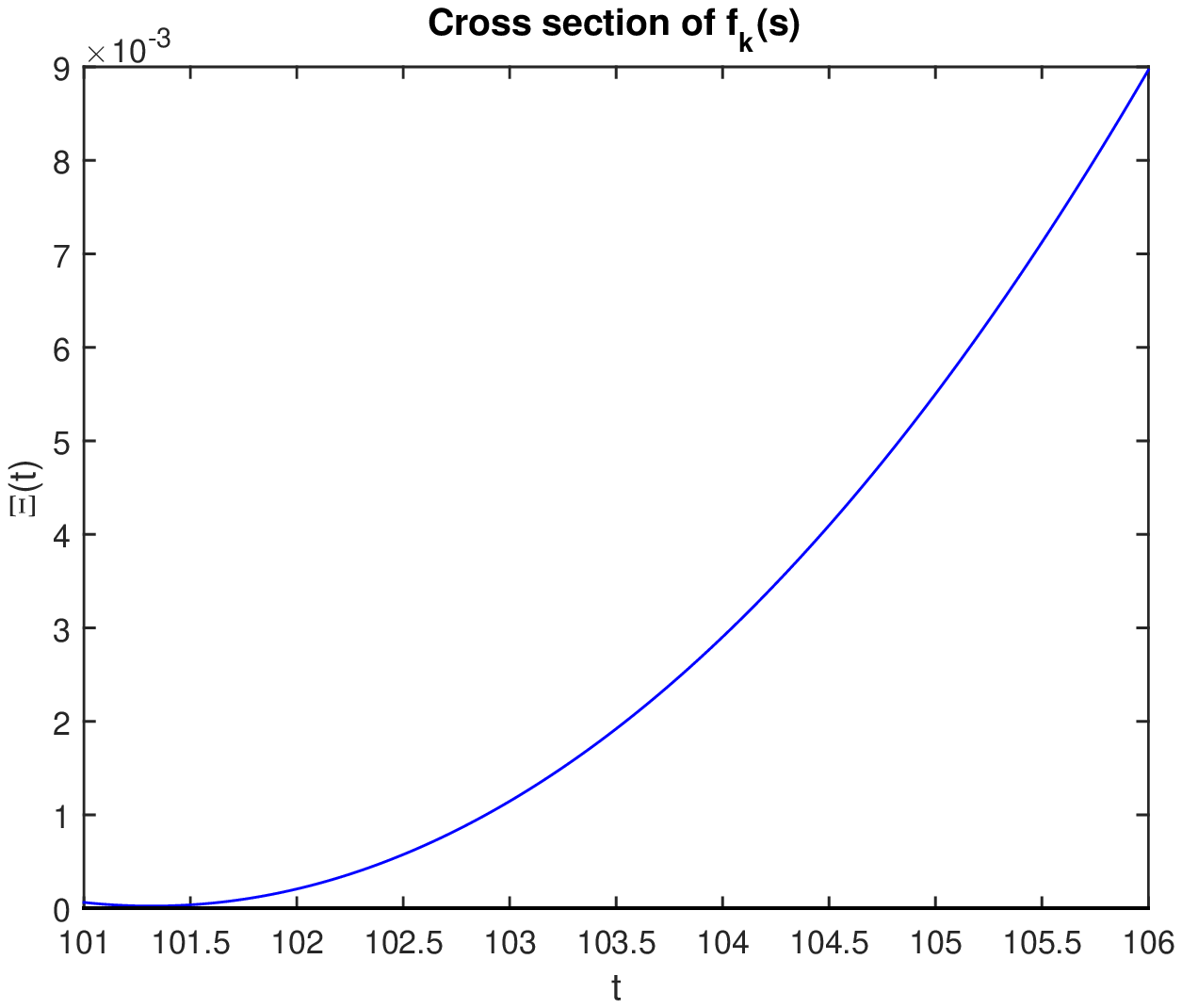}\\
   {\sf(b)}
  \end{minipage}
\begin{minipage}{.5\textwidth}
  \centering
  \includegraphics[scale=0.6]{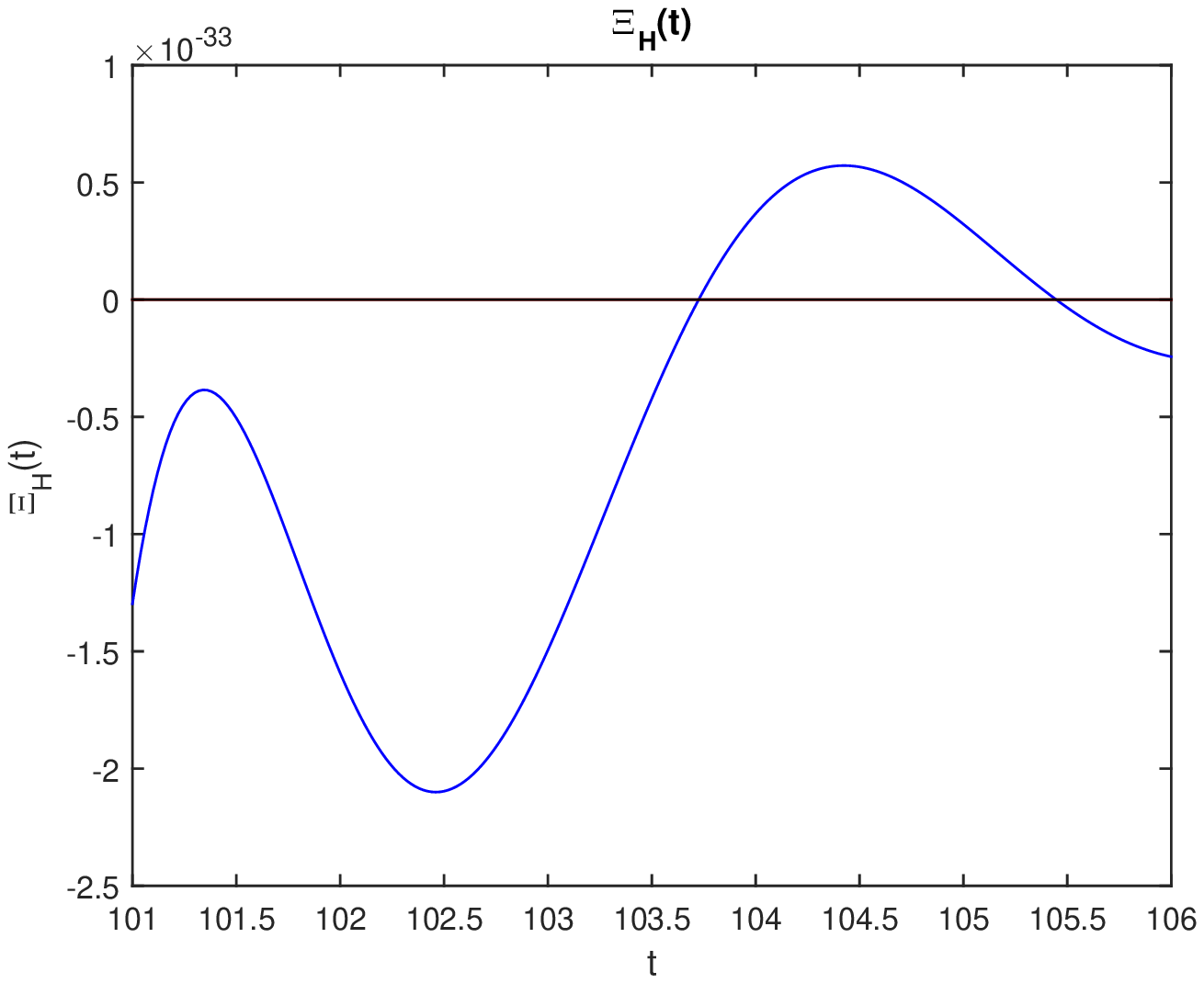}\\
   {\sf(c)}
 \end{minipage}%
 \begin{minipage}{.5\textwidth}
  \centering
  \includegraphics[scale=0.6]{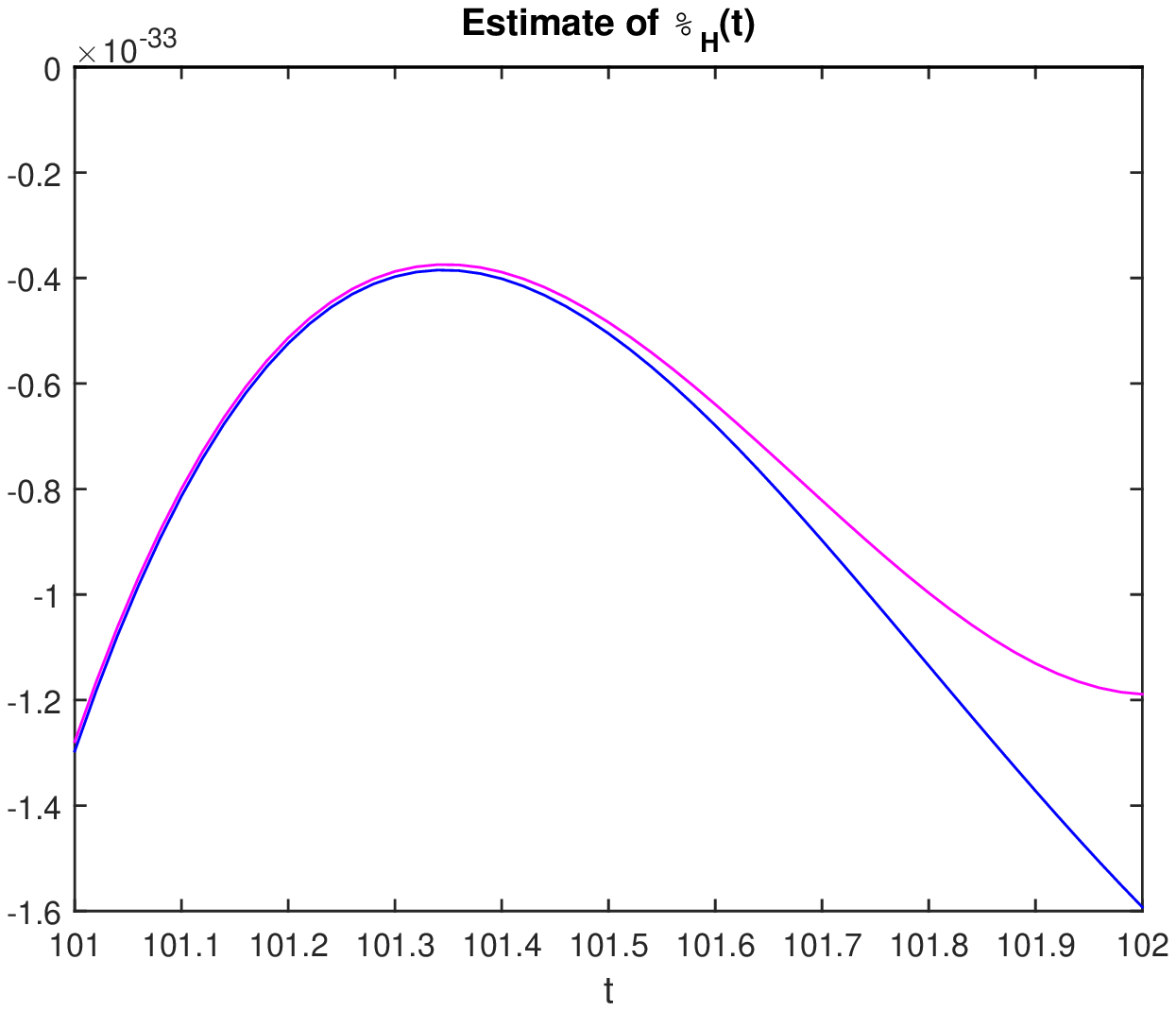}\\
   {\sf(d)}
  \end{minipage}
\caption{\sf (a) $\Xi_{(30)}(t)$, the first term in (\ref{Xi-product-form}); (b) $f_k(\tfrac{1}{2}+it)$, the second term in RHS of (\ref{Xi-product-form}); (c) $\Xi_H(t)$ of (\ref{Xi-product-form}); (d) $\Xi_H(t)$ (in blue) of (\ref{Xi-product-form}), and its approximation $\hat{\Xi}_H(t)$ (magenta) of (\ref{approx_Xi-product-form}) in the vicinity of $t_{30}=101.3178\ldots$.}
\label{fig-Xi-hypo}
\end{figure}

Figure \ref{fig-Xi-hypo} (a) shows $\Xi_{(30)}(t)$ of (\ref{xi{(k)}}), or the first term in the right-hand side (RHS) of (\ref{Xi-product-form}); Figure \ref{fig-Xi-hypo} (b) shows $f_k(\tfrac{1}{2}+it)$, the second term in RHS of (\ref{Xi-product-form}); Figure \ref{fig-Xi-hypo} (c) is $\Xi_H(t)$, the product of the complex functions plotted in (a) and (b); finally Figure \ref{fig-Xi-hypo} (d) show a zoomed-in version of $\Xi_H(t)$ (in blue)(i.e., for the interval $101\leq t \leq 102$), and its approximation $\hat{\Xi}_H(t)$ (in magenta) of $\Xi_{(30)}(t)$ obtained by using the approximation (\ref{xi-k-approx}).
A local maximum close to $t_{30}=101.3178\ldots$ is found from (\ref{delta-in-theorem}) at
\[ t^*=t_{30}-\delta_{30}=101.3178-\frac{3\lambda_{30}^2}{2t_{30}}=101.3178-9.2531\times 10^{-4}=101.3169.\]
From Figure \ref{fig-Xi-hypo}(c) (and its zoomed-in version (d)), we see that the function $\Xi_(t)$ takes its local maximum at $t^*\approx t_{30}$. Its value $\Xi_H(t^*)$ is strictly negative $\approx -0.4$, as seen from both (c) and (d), as well as
from Figure \ref{fig-hypothetical}(d), where it is a local minimum of the cross-section across $\lambda$ or $\sigma$.  This negativity of the local maximum is confirmed by evaluating (\ref{Xi-t_k}), (\ref{Xi'-to-Xi}) and (\ref{Xi''-to-Xi}):
\begin{align*}
\hat{\Xi}_H(t_{30})&=-4\alpha_{30}\lambda_{30}^2t_{30}^{-2}(t_{30}^2-\gamma_{30})
  =-3.8970\times 10^{-34},   \\
\frac{\hat{\Xi}_H'(t_{30})}{\hat{\Xi}_H(t_{30})}&=
  -t_{30}^{-1}\frac{3t_{30}^2-\gamma_{30}}{t_{3}^2-\gamma_{30}}= -0.8879,\\
\frac{\hat{\Xi}_H''(t_{30})}{\hat{\Xi}_H(t_{30})}&=
\left(\frac{\hat{\Xi}_H'(t_{30})}{\hat{\Xi}_H(t_{30})}\right)^2
+\frac{2}{\lambda_{30}^2}-\frac{4t_{30}^2}{t_{30}^2-\gamma_{30}}= 0.7884+32-0.8060=31.9824.
\end{align*}


\section{Concluding Remarks}

In this paper we have presented a new and promising direction towards a possible proof of the Riemann hypothesis, i.e., to show that all local maxima of $\Xi(t)$ are positive and all local minima are negative.  A further investigation along this line of argument will be reported in a forthcoming paper.\\

\noindent
\textbf{Acknowledgments:}  The author is very much indebted to Dr. Pei Chen and Prof. Brian L. Mark for their valuable advices, which helped the author resolve numerous problems he encountered in using MATLAB to obtain the numerical results presented in this paper.

\appendix
\numberwithin{equation}{section}
\renewcommand\thefigure{\thesection.\arabic{figure}}
\counterwithin{figure}{section}
\emph{}

\section{$\mathbf{\xi_{(j)}(s)}$ of (\ref{xi-minus-j})}\label{append-xi-setminus}

The function $\xi_{(j)}(s)$ is defined by (\ref{xi-minus-j}), i.e.,
\begin{align}\label{xi_j-product}
\xi_{(j)}(\tfrac{1}{2}+\lambda+it)=\tfrac{1}{2}\prod_{n\neq j, n>0} \frac{\lambda^2+t_n^2-t^2+2it\lambda}{t_n^2+\frac{1}{4}}.
\end{align}

Let us consider the following three different regimes:
\begin{enumerate}
\item $t_n\gg t\approx t_j$: There are infinitely many such $t_n$'s, but each term of the product form (\ref{xi_j-product}) can be approximated by a real constant, which converges to 1 as $t_n\to\infty$:
\[ \frac{\lambda^2+t_n^2-t^2+2it\lambda}{t_n^2+\frac{1}{4}}\approx 1-\frac{t^2-\lambda^2}{t_n^2}+2i\frac{t\lambda}{t_n^2}\to 1, ~~{\rm as}~~ t_n\to\infty.\]
For instance, for $t\approx t_{10}=49.7738\ldots$ (i.e., $j=10$), the term contributed by $n=30$ (where $t_{30}=101.3178\ldots$) can be approximated by a real constant: the real part of the above expression is $\approx 0.76$, and the imaginary part is within $\pm 5\times 10^{-3}$.  For $n=100$ ($t_{100}=236.7818\ldots$), the real part is $\approx 0.96$ and the imaginary part is much smaller, i.e., within the range of $\pm 0.9\times 10^{-3}$ for $|\lambda|<\tfrac{1}{2}$.

If we set   $j=30$, the term contributed by $n=50$ or $t_{50}=143.1118$ have its real part $\approx 0.5$ and
the imaginary part within $\pm 5\times 10^{-3}$. The term due to $n=200$, or $t_{200}=396.3818$, has the real part $\approx 0.93$ and the imaginary part within $\pm 6.4\times 10^{-4}$.

\item $t_n\sim t\approx t_j$:  Since the distance between two adjacent zeros is on the order of $O(1)$, $\lambda^2$ is negligibly smaller than the rest of the terms.  Thus,
\begin{align}
  \frac{\lambda^2+t_n^2-t^2+2it\lambda}{t_n^2+\frac{1}{4}}
 \approx\frac{t_n^2-t^2}{t_n^2}\left(1+\frac{2it\lambda}{t_n^2-t^2}\right).
 \end{align}
By setting $\displaystyle{a_n(t)=\frac{2t}{t_n^2-t^2}}$, each term is proportional to $(1-ia_n(t)\lambda)$, where $a_n(t)=O(t^{-1})\ll 1$.  Then, the following formula
\begin{align}
&(1+ia_1\lambda)(1+ia_2\lambda)=1-a_1a_2\lambda^2+i(a_1+a_2)\lambda,\nonumber\\
&(1+ia_1\lambda)(1+ia_2\lambda)(1+ia_3\lambda)=1-(a_1a_2+a_2a_3+a_3a_1)\lambda^2+a_1a_2a_3\lambda^4
+i[(a_1+a_2+a_3)\lambda-a_1a_2a_3\lambda^3],\nonumber\\
&(1+ia_1\lambda)(1+ia_2\lambda)(1+ia_3\lambda)(1+ia_4\lambda)=1-(\sum_{i\neq j}a_ia_j)\lambda^2+a_1a_2a_3a_4\lambda^4+i \left[(\sum_n a_n)\lambda-\sum_{i\neq j\neq k}a_ja_ja_k \lambda^3\right],\nonumber\\
& \cdots
\end{align}
Since $a_n\ll 1$ it is clear that the real and imaginary parts of the product can be approximated by a quadratic function and a linear function of $\lambda$, respectively.

\item $t_n\ll t\approx t_j$:  For given $t$, there is only a finite number of terms in this category.
\begin{align} \label{case-3}
  \frac{\lambda^2+t_n^2-t^2+2it\lambda}{t_n^2+\frac{1}{4}}\approx -\frac{t^2}{t_n^2}\left(1-\frac{t_n^2+\lambda^2}{t^2}-\frac{2i\lambda}{t}\right)\approx -\frac{t^2}{t_n^2},
\end{align}
where in the last expression we dropped the imaginary part because $t\gg t_1\gg
\tfrac{1}{2}>|\lambda|$. In this case, the imaginary part is much smaller than the real part.  In the case of $j=30$ (i.e., $t_{30}=101.3178\ldots$, the term for $n=1$ ($t_1=14.1347\ldots$) contributes
$\approx -5.0 \pm i 0.5$, which means that the real part is constant -50 and the imaginary part is a straight line whose maximum is 0.5 at $\lambda=\tfrac{1}{2}$. Similarly we find that the term for $n=10$ (i.e., $t_{10}=49.7738\ldots$) contributes $\approx -3.15 \pm i 0.04$. For $n=20$ (i.e., $t_{20}=77.1448$) the contribution to the product form is $\approx -0.33\pm i 0.017$.  The magnitude of the real part is less than one: this is because the last approximation in (\ref{case-3}) does not hold, because $t_n\ll t_j$ does not hold, although $t_n<t_j$. Nevertheless, this term still contributes as a multiplying real constant in the range $|\lambda|<1$, since the imaginary part is smaller by more than two orders of magnitude.

\end{enumerate}

From the above results, we can conclude that each term in the product form expression (\ref{xi_j-product}) is a real constant ($<1$) when  $t_n\gg t$ (i.e., Case 1); it is a negative constant, whose magnitude is greater than unity, when $t_n \ll t$ (i.e., Case 3).  Consequently, the functional form of $\xi_j(\tfrac{1}{2}+\lambda+it)$ of (\ref{xi_j-product}) is dictated by a finite number of terms whose $t_n$ values are not far from $t$ (i.e., Case 2). It is not difficult to see that a holomorphic function whose cross section along $t=t_j$ has a quadratic function in its real part and a straight line in its imaginary part that cuts through zero is limited to the following function
\begin{align}\label{shape-of-xi_j}
 \alpha_j(\lambda^2-t^2+\gamma_j+2it\lambda).
\end{align}
where $\alpha_j$ and $\gamma_j$ are real constants. It is not difficult to see that this function satisfies the Cauchy-Riemann equations, and Laplace's equation. Furthermore, the function is reflective, just like $\xi(s)$ and $(s-\rho_j)(s-\rho_{j'})$.  If we should include $\lambda^4$ and higher-order terms in the real part, and $\lambda^3$ and nigher order terms in the imaginary part, we could certainly improve the approximation, but for the purpose of our analysis, such generalization would merely complicate the analysis and would not contribute to any additional insight into the heart of the problem we are interested in.

On the critical line $\lambda=0$, each term in the product representation (\ref{xi-minus-j}) is given by
\[ \frac{\lambda^2+t_n^2-t^2+ i 2t\lambda}{t_n^2+\tfrac{1}{4}}=\frac{t_n^2-t^2}{t_n^2+\tfrac{1}{4}}.\]
Hence, the function $\xi_{(j)}(s)$ is real on the critical line, and we denote it by $\Xi_j^*(t)$:
\begin{align}
\Xi_j^*(t)&=\xi_{(j)}(\tfrac{1}{2}+it)
=\tfrac{1}{2}\prod_{n\neq j, n>0}\frac{t_n^2-t^2}{t_n^2+\tfrac{1}{4}}
=\frac{t_j^2+\frac{1}{4}}{t_j^2-t^2} \Xi(t),
\end{align}
where $\Xi(t)$ is defined by (\ref{def-Xi}), which can be written because of the product form (\ref{product-form}) as
\begin{align}\label{Xi-product}
\Xi(t)=\tfrac{1}{2}\prod_n\left(1-\frac{s}{\rho_n}\right)
=\tfrac{1}{2}\prod_{n>0}\frac{\lambda^2+t_n^2-t^2+2it\lambda}{t_n^2+\tfrac{1}{4}}.
\end{align}

The slope of the imaginary part of $\xi(\tfrac{1}{2}+\lambda+it_j)$ crosses the critical line $\lambda=0$ is given by $b(t_j)=-\Xi'(t_j)$ as given in (\ref{def-b-t}).  From (\ref{def-b-t}), we have
\begin{align}
b(t)&=-\Xi'(t)=-2\left(\sum_{n>0}\frac{t}{t-t_n^2}\right) \Xi(t)
=2\left(\sum_{n>0}\frac{t}{t_n^2-t^2}\right)\left(\frac{t_j^2-t^2}{t_j^2+\tfrac{1}{4}}\right)\Xi_j^*(t).
\end{align}
If we set $t=t_j$ all terms in the summation vanish except for the term $n=j$, thus obtaining
\begin{align}
b(t_j)&=\frac{2t_j}{t_j^2+\tfrac{1}{4}}\Xi_j^*(t_j)\approx \frac{2\Xi_j^*(t_j)}{t_j},
\end{align}
which can be alternatively derived by differentiating (\ref{Im-xi}) w.r.t. $\lambda$ and setting $\lambda=0$.

\section{Derivation of an extremum point near $t=t_k$ in the function $\Xi(t)$}

When the imaginary part of the cross section $\xi(\tfrac{1}{2}+\lambda+it_k)$ crosses zero at $\lambda=-\lambda_k, 0$ and $\lambda_k$ for some $\lambda_k$ such that $0<\lambda_k<\tfrac{1}{2}$, which is a necessary condition to disprove the Riemann hypothesis, the slope $b(t)$ at $t=t_k$ is not zero. In other words, $t=t_k$ is not an extremum point of the real function $\Xi(t)$.  In this appendix we derive an extremum point which is close to
$t=t_k$.

By applying the result obtained in Appendix A, we find that the $\xi_{(k)}(s)$ takes the following form for $t_k\gg 1$:
\[ \xi_{(k)}(\tfrac{1}{2}+\lambda+it_k)= \alpha_k(\lambda^2-t^2+\gamma_k+2it\lambda),~~{\rm for}~~t\approx t_k,\]
where $\alpha_k$ and $\gamma_k$ are real constants.  Thus, we have
\begin{align}\label{four-factors-3}
\xi(s)=t_k^{-4}[\lambda^2-\lambda_k^2-(t-t_k)^2+i2(t-t_k)\lambda]\cdot[\lambda^2-\lambda_k^2-(t+t_k)^2+i2(t+t_k)\lambda]
\cdot \alpha_k(\lambda^2-t^2+\gamma_k+2it\lambda).
\end{align}
The imaginary part of the above function is
\begin{align}
\alpha_k^{-1}t_k^4\cdot\Im\{\xi(s)\}&=  2(t-t_k)\lambda(\lambda^2-\lambda_k^2-(t+t_k)^2)(\lambda^2-t^2+\gamma_k)
             +2(t+t_k)\lambda(\lambda^2-\lambda_k^2-(t-t_k)^2)(\lambda^2-t^2+\gamma_k)\nonumber\\
&~~~~~ +2t\lambda\left[(\lambda^2-\lambda_k^2-(t-t_k)^2)
   (\lambda^2-\lambda_k^2-(t+t_k)^2)-4\lambda^2(t^2-t_k^2)\right]\label{Imxi-s}
\end{align}
The cross-section of $\Im\{\xi(s)\}$ on the line $t=t_k\gg 1$ is
\begin{align}\label{imag-xi-t_k}
\alpha_k^{-1}t_k^4\cdot\Im\{\xi(\tfrac{1}{2}+\lambda+it_k)\}
&=4t_k(\gamma_k-t_k^2+\lambda^2)\lambda(\lambda^2-\lambda_k^2) +4t_k (-4t_k^2+\lambda^2-\lambda_k^2)\lambda(\lambda^2-\lambda_k^2)\nonumber\\
 &\approx 4t_k(\gamma_k-5t_k^2)\lambda(\lambda^2-\lambda_k^2),
\end{align}
which is a cubic function, and crosses zero at $\lambda=-\lambda_k, 0$ and $\lambda_k$.
The critical points, where the slope $b(\lambda, t_k)=0$ (see (36) in \cite{report-no-5}), are at $\displaystyle{\lambda=\pm\frac{\lambda_k}{\sqrt{3}}}$.

By taking the partial derivative of the above equation w.r.t. $\lambda$ and setting $\lambda=0$, we find
\begin{align}\label{b-t-result}
b(t)=2\alpha_kt_k^{-4}t\left\{2(\lambda_k^2+t^2-t_k^2)(t^2-\gamma_k)
+[\lambda_k^2+(t-t_k)^2][\lambda_k^2+(t+t_k)^2]\right\}.
\end{align}
We should note that the above result can be alternatively derived from $\Xi(t)$, using the formulas  (\ref{def-b-t}).  The function $\Xi(t)$ can be simply obtained by setting $\lambda=0$ in (\ref{xi_H}).  The value of $b(t)$ at $t=t_k$ can be obtained as
\begin{align}\label{b(t_k)}
b(t_k)=2\alpha_k\lambda_k^2t_k^{-3}(6t_k^2-2\gamma_k+\lambda_k^2)=4\alpha_k\lambda_k^2 t_k^{-3}(3t_k^2-2\gamma_k)+O(t_k^{-2}),
\end{align}
which is clearly not zero, although it converges to zero as $t_k\to\infty$.

In order to find $t^*$ in the vicinity of $t_k$ such that $b(t^*)=-\Xi'(t^*)=0$, we set (\ref{b-t-result}) to zero, obtaining
\begin{align}
3t^4+2(2\lambda_k^2-2t_k^2-\gamma_k)t^2+2\gamma_k(t_k^2-\lambda_k^2)+(\lambda_k^2+t_k^2)^2=0,
\end{align}
from which we find the extremum point that we are after:
\begin{align}
{t^*}^2&=\tfrac{1}{3}\left(2t_k^2+\gamma_k-2\lambda_k^2+\sqrt{t_k^4-2(\gamma_k+7\lambda_k^2)t_k^2+(\lambda_k^2-\gamma_k)^2}\right)\nonumber\\
&=\tfrac{1}{3}\left(2t_k^2+\gamma_k-2\lambda_k^2+t_k^2-(\gamma_k+7\lambda_k^2)+O(t_k^{-2})\right)=t_k^2-3\lambda_k^2+O(t_k^{-2}).
\end{align}
Thus, the function $\Xi(t)$ takes a local extremum at $t=t^*$, where
\begin{align}\label{t-star}
t^*=t_k\sqrt{1-\frac{3\lambda_k^2}{t_k^2}+O(t_k^{-4})}=t_k-\delta_k,
\end{align}
where
\begin{align}\label{delta}
\delta_k=\frac{3\lambda_k^2}{2t_k}+O(t_k^{-2}).
\end{align}
The cross-section function $\Re\{\xi(\tfrac{1}{2}+\lambda+it^*)\}$ should be a quadratic function of $\lambda$ of the form $a(t^*)\lambda^2+\cdots$, where the function $a(t)$ can be readily found from $b(t)$:
\begin{align}
a(t)&=\tfrac{1}{2}b'(t)=\alpha_kt_k^{-4}\left\{2[\lambda_k^2t^2+(t^2-t_k^2)t^2-\gamma_k(t^2-t_k^2+\lambda_k^2)]
+[\lambda_k^2+(t-t_k)^2][\lambda_k^2+(t+t_k)^2]\right.\nonumber\\
&~~~~~~~~~~~~~~~~~+\left.4t^2(2\lambda_k^2+3t^2-2t_k^2-\gamma_k)\right\}\nonumber\\
&=\alpha_kt_k^{-4}\left[15t^4-12(t_k^2-\lambda_k^2)t^2+(\lambda_k^2+t_k^2)^2-2\gamma_k(3t^2-t_k^2+\lambda_k^2)\right]\nonumber\\
&=\alpha_kt_k^{-4}\left[15t^4-6(2t_k^2+\gamma_k-2\lambda_k^2)t^2+(t_k^2+\lambda_k^2)^2+2\gamma_k(t_k^2-\lambda_k^2)\right].
\end{align}

At $t=t_k$, its value is
\begin{align}
a(t_k)&=\alpha_kt_k^{-4}\left[4t_k^4-2(2\gamma_k-7\lambda_k^2)t_k^2-(2\gamma_k-\lambda_k^2)\lambda_k^2\right]\nonumber\\
&=4\alpha_k\left(1-\frac{2\gamma_k-7\lambda_k^2}{2t_k^2}+O(t_k^{-4})\right).
\end{align}

From Lemma \ref{lemma-saddle-point} we know that $s=\tfrac{1}{2}+it^*$ is a saddle point.  Thus, $\Xi(t)$ can be approximated by a quadratic function around $t=t^*$, i.e.,
\begin{align}\label{Xi-t-quadratic}
\Xi(t)= -a(t^*)(t-t^*)^2+\cdots,~~{\rm for}~~t\approx t^*.
\end{align}
Thus,
\begin{align}
\Xi(t_k)-\Xi(t^*)&=-a(t^*)\delta_k^2=-\frac{36\alpha_k\lambda_k^4}{4t_k^2}+O(t_k^{-4}),
\end{align}
where we used the estimate of $a(t^*)$ given by
\begin{align}
a(t^*)&=a(t_k)-a'(t_k)\delta_k+O(\delta_k^2)=4\alpha_k-\alpha_kt_k^{-4}60t_k^3\delta_k+O(\delta_k^2)\nonumber\\
&=4\alpha_k\left(1-\frac{45\lambda_k^2}{t_k^2}\right)+O(t_k^{-2}).
\end{align}
The slope of the tangent of $\Xi(t_k)$ is obtained from (\ref{Xi-t-quadratic}) as
\begin{align}
\frac{d\Xi(t_k)}{dt}=-2a(t^*)(t_k-t^*)=-2a(t^*)\delta_k=-12\alpha_k\lambda_k^2 t_k^{-1}.
\end{align}
From the Cauchy-Riemann equation,
\begin{align} \frac{d\Xi(t_k)}{dt}=-\left.\frac{\partial\Im\{\xi(\tfrac{1}{2}+\lambda+it_k)\}}{\partial\lambda}\right|_{\lambda=0}
=-b(t_k),
\end{align}
which confirms the result (\ref{b(t_k)}).

\end{document}